\documentclass[12pt]{amsart}

\usepackage{amssymb}

\newtheorem{theorem}{Theorem}[section]
\newtheorem{proposition}[theorem]{Proposition}
\newtheorem{lemma}[theorem]{Lemma}
\newtheorem{corollary}[theorem]{Corollary}

\theoremstyle{definition}

\newtheorem{example}[theorem]{Example}

\newtheorem{remark}[theorem]{Remark}

\newcommand{\ZZ}{ \ensuremath{\mathbb{Z}}}

\newcommand{\PP}{\mathbb{P}}

\newcommand{\rank}{\ensuremath{\mathrm{rank}}\hspace{1pt}}

\newcommand{\Ker}{\ensuremath{\mathrm{Ker}}\hspace{1pt}}

\newcommand{\mideal}{\ensuremath{\mathfrak{m}}}

\def\cocoa{{\hbox{\rm C\kern-.13em o\kern-.07em C\kern-.13em o\kern-.15em A}}}

\newcommand{\MM}{\mathcal{M}}

\begin{document}

\title{On ideals with the Rees property}

\author[J. Migliore]{Juan Migliore}
\address{
Department of Mathematics, University of Notre Dame, Notre Dame, IN 46556, USA
}
\email{Juan.C.Migliore.1@nd.edu
}

\author[R.M. Mir\'o-Roig]{Rosa M. Mir\'o-Roig}
\address{
Facultat de Matem\`atiques, Department df\`Algebra i Geometria, Gran Via des les Corts
Catalanes 585, 08007 Barcelona, SPAIN
}
\email{miro@ub.edu}

\author[S. Murai]{Satoshi Murai}
\address{
Department of Mathematical Science,
Faculty of Science,
Yamaguchi University,
1677-1 Yoshida, Yamaguchi 753-8512, Japan.
}
\email{murai@yamaguchi-u.ac.jp}

\author[U. Nagel]{Uwe Nagel}
\address{
Department of Mathematics, University of Kentucky, 715 Patterson Office Tower,
Lexington, KY 40506-0027, USA
}
\email{uwe.nagel@uky.edu}

\author[J. Watanabe]{Junzo Watanabe}
\address{
Department of Mathematics,
Tokai University,
Hiratsuka 259-1292, Japan
}
\email{watanabe.junzo@tokai-u.jp}

\thanks{
Part of the work for this paper was done while the first author was sponsored by the National Security Agency under Grant Number H98230-12-1-0204.
Part of the work for this paper was done while the second author was sponsored by the Grant MTM2010-15256.
Part of the work for this paper was done while the third author was sponsored by KAKENHI 22740018.
Part of the work for this paper was done while the fourth author was sponsored by the National Security Agency under Grant Number H98230-12-1-0247.
This work was started at the workshop ``Aspects of SLP and WLP,"
held at Hawaii Tokai International College (HTIC) in September 2012,
which was financially supported by JSPS KAKENHI 24540050.
The authors thank HTIC for their kind hospitality.
Also, we would like to thank Mats Boij for valuable discussions at the workshop.}
%Keyword and Subject Classes (if needed)
%\keywords{weak Lefschetz property, Rees property, m-full ideals, Sperner property, monomial ideals, almost complete intersection, Hilbert function, LYM property.}
%\dedicatory{Dedicated to on the occasion of his birthday}

\begin{abstract}
A homogeneous ideal $I$ of a polynomial ring $S$ is said to have the Rees property if,
for any homogeneous ideal $J \subset S $ which contains $I$, the number of generators of $J$ is smaller than or equal to that of $I$.
A  homogeneous ideal $I \subset S$ is said to be $\mathfrak m$-full if $\mathfrak mI:y=I$ for some $y \in \mathfrak m$,
where $\mathfrak m$ is the graded maximal ideal of $S$.
It was proved by one of the authors that $\mathfrak m$-full ideals have the Rees property
and that the converse holds in a polynomial ring with two variables.
In this note, we give examples of ideals which have the Rees property but are not $\mathfrak m$-full
in a polynomial ring with more than two variables.
To prove this result, we also show that every Artinian monomial almost complete intersection in three variables has the Sperner property.
\end{abstract}

\maketitle

\section{Introduction}
Let $(R,\mideal,K)$ be a local ring.
An ideal $I \subset R$ is said to have the {\em Rees property} if, for any ideal $J \supset I$,
one has $\mu(J) \leq \mu(I)$,
where $\mu(L)=\dim_K L/\mideal L$ denotes the number of minimal generators of an ideal $L \subset R$.
An ideal $I \subset R$ is said to be {\em $\mideal$-full} if there is an element $y \in \mideal$ such that $\mideal I : y = I$.
It was proved in \cite[Theorems 3 and 4]{Wa} that $\mideal$-full ideals have the Rees property,
and that the converse is also true if $R$ is a regular local ring of dimension $2$.
It was asked by Harima et al.\ \cite[Question 2.56]{Book} if, for a regular local ring $R$ of characteristic $0$,
ideals with the Rees property are precisely $\mideal$-full ideals.
The purpose of this note is to prove that, in a polynomial ring with more than $2$ variables, there are infinitely many ideals which have the Rees property but are not $\mideal$-full.

The Rees property and the $\mideal$-full property have close connections to the Sperner property and the weak Lefschetz property.
Let $S=K[x_1,x_2,\dots,x_n]$ be a standard graded polynomial ring and $\mideal=(x_1,\dots,x_n)$ the graded maximal ideal of $S$.
The Rees property and m-fullness are defined analogously for homogenous  ideals $I$ of $S$.
An Artinian graded $K$-algebra $A=S/I$ is said to have the {\em Sperner property} if
$$ \max  \{ \mu(J): J \mbox{ is an ideal of }A\}= \max \{ \dim_K A_k: k \in \ZZ_{\geq 0}\}.$$
Note that, in general, the left-hand side in the above equation is larger than or equal to the right-hand side
since $\max \{ \dim_K A_k: k \in \ZZ_{\geq 0}\}=\max\{\mu (\mideal^k A): k \in \ZZ_{\geq 0} \} $.
This notion comes from the Sperner property in the theory of partially ordered sets.
See \cite[\S 1]{Book} for details.
An Artinian graded $K$-algebra $A$ is said to have the {\em weak Lefschetz property} ({\em WLP} for short)
if there is a linear form $y \in A$ such that the multiplication
$$\times y : A_{k-1} \to A_k$$
is either injective or surjective for all $k$.
It is known that the WLP implies the Sperner property (see \cite[Proposition 3.6]{Book}).
On the other hand,
we prove in Proposition \ref{lemD} that
if $A=S/I$ has the Sperner property and $p$ is an integer with $\dim_K A_p = \max \{ \dim_K A_k: k \in \ZZ_{\geq 0}\}$,
then the ideal $I+\mideal^p$ has the Rees property.
Moreover, we prove that if there is an integer $s <p$ such that $I$ has no generators of degree $s+1$ and
the multiplication $\times y : A_s \to A_{s+1}$ is not injective for any linear form $y$,
then the ideal $I+\mideal^p$ is not $\mideal$-full.
By virtue of this result, to find ideals which have the Rees property but are not $\mideal$-full,
we may study Artinian graded $K$-algebras which have the Sperner property but do not have the WLP because of the failure of the injectivity.
To find such graded $K$-algebras, we consider monomial almost complete intersections.

For integers $a,b,c,\alpha,\beta,\gamma \in \ZZ_{\geq 0}$,
let
\begin{align}
\label{1}
I_{a,b,c,\alpha,\beta,\gamma}=(x_1^a, x_2^b, x_3^c, x_1^\alpha x_2^\beta x_3^\gamma) \subset K[x_1,x_2,x_3],
\end{align}
where $\alpha<a, \beta<b, \gamma<c$ and where at least two of $\alpha,\beta,\gamma$ are nonzero.
We call the ideal $I_{a,b,c,\alpha,\beta,\gamma}$
a {\em monomial almost complete intersection ideal} in $K[x_1,x_2,x_3]$.
About these ideals, we first prove the following result.

\begin{theorem}
\label{thm1}
The $K$-algebra $K[x_1,x_2,x_3]/I_{a,b,c,\alpha,\beta,\gamma}$ has the Sperner property.
\end{theorem}

It was proved in \cite{CN} and \cite{MMN} that if $A=K[x_1,x_2,x_3]/I_{a,b,c,\alpha,\beta,\gamma}$ fails to have the WLP in characteristic $0$,
then $a+b+c+\alpha + \beta + \gamma$ is divisible by $3$
and the multiplication
$$\times y : A_{\frac 1 3 (a+b+c+\alpha + \beta + \gamma)-2} \to A_{\frac 1 3 (a+b+c+\alpha + \beta + \gamma)-1}$$
is not injective for any linear form $y$ (see Lemma \ref{lemC}).
By using this failure of the injectivity,
we prove the following theorem.

\begin{theorem}
\label{thm2}
%Let
%$I_{a,b,c,\alpha,\beta,\gamma}$ be a monomial almost complete intersection ideal as in \eqref{1} and $s=\frac 1 3 (a+b+c+\alpha + \beta + \gamma)-2$.
%If $K[x_1,x_2,x_3]/I_{a,b,c,\alpha,\beta,\gamma}$ does not have the WLP and $I_{a,b,c,\alpha,\beta,\gamma}$ has no generators of degree $s+1$,
%then the ideal $I_{a,b,c,\alpha,\beta,\gamma}+ \mideal^{s+1}$ has the Rees property but is not $\mideal$-full.
If $K[x_1,x_2,x_3]/I_{a,b,c,\alpha,\beta,\gamma}$ does not have the WLP and $I_{a,b,c,\alpha,\beta,\gamma}$ has no generators of degree $s+1$,
where $s=\frac 1 3 (a+b+c+\alpha + \beta + \gamma)-2$,
then the ideal $I_{a,b,c,\alpha,\beta,\gamma}+ \mideal^{s+1}$ has the Rees property but is not $\mideal$-full.
\end{theorem}

To apply the above theorem,
one needs to know when a monomial almost complete intersection ideal in $K[x_1,x_2,x_3]$ fails to have the WLP.
Such a problem was studied in \cite{CN} and \cite{MMN}.
In particular, classes of monomial almost complete intersection ideals in $K[x_1,x_2,x_3]$ which fail to have the WLP
were provided in \cite[Corollary 7.4]{MMN} and in \cite[Theorem 10.9(b)]{CN}.
These results and Theorem \ref{thm2} provide many ideals which have the Rees property but are not $\mideal$-full.
Here we give one simple family of such ideals.

\begin{example} Let $k$ and $\alpha$ be odd integers with $k \geq 2 \alpha +3$ and with $\alpha \geq 3$,
and let
$$I_{k,\alpha}=(x_1^k,x_2^k,x_3^k,x_1^\alpha x_2^\alpha x_3^\alpha) \subset K[x_1,x_2,x_3]. $$
It was proved in \cite[Corollary 7.6]{MMN} that $K[x_1,x_2,x_3]/I_{k,\alpha}$ fails to have the WLP in any characteristic.
Since $I_{k,\alpha}$ has no generators of degree $s+1 =k+\alpha-1$,
the ideal $I_{k,\alpha}+(x_1,x_2,x_3)^{k+\alpha-1}$ has the Rees property but is not $\mideal$-full
by Theorem \ref{thm2}.
\end{example}

It is also possible to construct ideals which have the Rees property but are not $\mideal$-full
in a polynomial ring with more than three variables from Theorem \ref{thm2}.
For example, if $J\subset K[x_1,x_2,x_3]$ is such an ideal given by Theorem \ref{thm2},
then the ideal  $J \cdot T +(x_1y,x_2y,x_3y,y^2)$ of the polynomial ring $T=K[x_1,x_2,x_3,y]$
has these properties.

About the Rees property, it is natural to consider the following stronger property.
We say that an ideal
$I \subset S$ has the {\em strong Rees
property} if, for any ideal $J \supsetneq I $,
one has $\mu(J) < \mu(I)$.
Ideals given by Theorem \ref{thm2} do not satisfy this stronger property (see Remark \ref{rem}). % at least if they also have no minimal generator in degree $s$. 
%Indeed,  \cite[Proposition 10.8]{CN} implies that for such an ideal $I_{a,b,c,\alpha,\beta,\gamma}$, the value of the Hilbert function of the quotient has equal values in degrees $s$ and $s+1$; thus, $J = I_{a,b,c,\alpha,\beta,\gamma} + \mathfrak m^s$ and $I = I_{a,b,c,\alpha,\beta,\gamma} + \mathfrak m^{s+1}$ have the same number of minimal generators.
However, in Theorem \ref{3.1}, we give infinitely many ideals which have the strong Rees property but are not $\mideal$-full
in the polynomial ring with more than $3$ variables.

\section{Proof of theorems}

Let $S=K[x_1,\dots,x_n]$.
A graded $K$-algebra $A=S/I$ is {\em Artinian} if $\dim_K A < \infty$. 
A {\em monomial Artinian $K$-algebra} is an Artinian graded $K$-algebra $A=S/I$ such that $I$ is a monomial ideal.
The {\em Hilbert function} of a graded $K$-algebra $A$ is the numerical function defined by
$$H(A,k)=\dim_K A_k$$
for $k \in \ZZ_{\geq 0}$.
We say that the Hilbert function of $A$ is {\em unimodal} if there is an integer $p$
such that $H(A,k-1) \leq H(A,k)$ for $k \leq p$ and $H(A,k) \geq H(A,k+1)$ for $k \geq p$.
For a monomial Artinian $K$-algebra $A=S/I$,
let $\MM_k(A)$ be the set of monomials of degree $k$ which are not in $I$,
and let
$$B_k(A)=\big\{ (m,m') \in \MM_{k-1}(A) \times \MM_k(A): m \mbox{ divides }m'\big\}.$$
A subset $M \subset B_k(A)$ is called a {\em full matching} of $B_k(A)$ if
\begin{itemize}
\item[(i)] $(f,g), (f',g') \in M$, $(f,g) \ne (f',g')$ imply $ f \ne f'$ and $g \ne g'$, and
\item[(ii)] $\#M= \min \{ \# \MM_{k-1}(A),\# \MM_k(A)\}$.
\end{itemize}
Here $\#X$ denotes the cardinality of a finite set $X$.
We say that $A$ has a {\em full matching at degree $k$} if there is a full matching of $B_k(A)$.

The following facts are standard in the study of the Sperner property.
See e.g., \cite[Theorem 1.31, Lemma 1.50 and Proposition 2.41]{Book}.

\begin{lemma}
\label{lemA}
If a monomial Artinian $K$-algebra $A$ has the unimodal Hilbert function and has a full matching at degree $k$ for all $k \in \ZZ_{\geq 0}$,
then $A$ has the Sperner property.
\end{lemma}

\begin{lemma}
\label{lemB}
Let $A$ be a monomial Artinian $K$-algebra and $k$ a positive integer.
If the multiplication $\times y:A_{k-1} \to A_k$ by a linear form $y$ is either injective or surjective then $A$ has a full matching at degree $k$.
\end{lemma}

The next results were proved in \cite{CN,MMN}.

\begin{lemma}
\label{lemC}
Let $I_{a,b,c,\alpha,\beta,\gamma}$ be a monomial almost complete intersection ideal as in \eqref{1}
and $s=\frac 1 3 (a+b+c+\alpha + \beta + \gamma )-2$.
If the characteristic of $K$ is zero and $A=K[x_1,x_2,x_3]/I_{a,b,c,\alpha,\beta,\gamma}$ does not have the WLP, then
\begin{itemize}
\item[(i)] $($\cite[Theorem 6.2]{MMN}$)$ $s$ is an integer.
\item[(ii)] $($\cite[Proposition 10.7]{CN}$)$ $\dim_K A_{s} = \dim_K A_{s+1} = \max\{ \dim_K A_k: k \in \ZZ_{\geq 0}\}$.
Also, the multiplication $\times y: A_s \to A_{s+1}$ is not injective for any linear form $y$ in $K[x_1,x_2,x_3]$.
\item[(iii)] $($\cite[Proposition 10.8]{CN}$)$ $A$ has a full matching at degree $s+1$.
\end{itemize}
\end{lemma}

We now prove Theorem \ref{thm1}.

\begin{proof}[Proof of Theorem \ref{thm1}]
%Since the Sperner property of a monomial $K$-algebra $A=S/I$ only depends on the poset structure of the monomial $k$-basis $\bigcup_{k \geq 0} \mathcal M_k(A)$ of $A$ (see Section 3 later),
%we may assume that the characteristic of $K$ is zero.
Since taking initial ideals does not decrease the number of minimal generators,
we only need to consider monomial ideals to prove the Sperner property of a monomial $K$-algebra.
Thus we may assume that the characteristic of $K$ is zero.
Let $I_{a,b,c,\alpha,\beta,\gamma}$ be a monomial almost complete intersection ideal as in \eqref{1}, $s= \frac 1 3 ( a+b+c+\alpha + \beta + \gamma)-2$
and $A=K[x_1,x_2,x_3]/I_{a,b,c,\alpha,\beta,\gamma}$.
If $A$ has the WLP then $A$ has the Sperner property by Lemmas \ref{lemA} and \ref{lemB}.
Suppose that $A$ does not have the WLP.
By Lemma \ref{lemC}(i) and (iii),
$s$ is an integer and $A$ has a full matching at degree $s+1$.

\noindent {\em Claim:} The multiplication
$\times (x_1+x_2+x_3) :  A_{k-1} \to A_k$
is injective for $k < s+1$ and is surjective for $k > s+1$.

\noindent {\em Proof of the Claim:} According to  \cite[Proposition 2.2]{MMN}, the Claim is true for a general linear form $L\in A$ if and only if it is true for $x_1+x_2+x_3$. Let us prove it for a general linear form $L\in A$. To this end, we consider $\mathcal E$  the syzygy bundle  of $I_{a,b,c,\alpha,\beta,\gamma}$, i.e.  $$\mathcal E:=\Ker
({\mathcal O}_{\PP^2}(-a)\oplus {\mathcal O}_{\PP^2}(-b)\oplus {\mathcal O}_{\PP^2}(-c)\oplus {\mathcal O}_{\PP^2}(-\alpha-\beta-\gamma)
\stackrel{(x_1^a, x_2^b, x_3^c, x_1^\alpha x_2^\beta x_3^\gamma)
}{\longrightarrow}
 {\mathcal O}_{\PP^2}) $$ and  $L \cong
\mathbb P^1$  a general line.    By \cite[Theorem 3.3]{BK}, if $I_{a,b,c,\alpha,\beta,\gamma}$
 fails to have the WLP then $\mathcal E$ is semistable.  Furthermore,  the
splitting type of ${\mathcal E}_{\rm norm}:={\mathcal E}(s+2)$ must be $(1,0,-1)$, i.e. $${\mathcal E}_{|L}(s+2)\cong {\mathcal O}_{L}(-1)\oplus {\mathcal O}_{L}\oplus {\mathcal O}_{L}(1)$$
(apply \cite[Theorem 2.2]{BK} and the Grauert-M\"ulich theorem). Therefore, we have
$H^0(L,{\mathcal E}_{|L}(t))=0$ for $t\le s$ and $H^1(L,{\mathcal E}_{|L}(t))=0$ for $t\ge s+2$ and we easily conclude using the exact cohomology sequence

{\small
$$\cdots \longrightarrow H^0(L,{\mathcal E}_{|L}(t))  \longrightarrow  H^1(\PP^2,{\mathcal E}(t-1))\stackrel{\times L
}{\longrightarrow}  H^1(\PP^2,{\mathcal E}(t))  \longrightarrow  H^1(L,{\mathcal E}_{|L}(t)) \longrightarrow \cdots $$ $$\hskip 11mm  \parallel \hskip 26mm \parallel $$ $$\hskip 9mm  A_{t-1} \hskip 22mm A_t $$
}
Finally, note that the above claim and Lemmas \ref{lemA} and \ref{lemB} prove the
desired statement. \end{proof}

\begin{remark}
%The desired 
The surjectivity of the Claim in the above proof has also  been shown in Proposition 9.7 of ``Enumerations deciding the weak Lefschetz property'', the 2011 predecessor of \cite{CN}.
\end{remark}

\begin{remark}
In the above proof, we saw that
any monomial almost complete intersection ideal in $K[x_1,x_2,x_3]$ {\em almost} has the WLP
in the sense that the multiplication of $x_1+x_2+x_3$ is either injective or surjective except for the map between degrees $s$ and $s+1$.
This property does not hold for all almost complete intersection monomial ideals.
For example, if $S=K[x_1,\dots,x_5]$ and $I =(x_1^5,\dots,x_5^5,x_1x_2x_3x_4x_5)$,
then the multiplication $\times (x_1+ \cdots + x_5): (S/I)_k \to (S/I)_{k+1}$ is neither injective nor surjective
for $k=8$ and $k=9$.
\end{remark}

Now Theorem \ref{thm2} follows from Lemma \ref{lemC}(ii) and the next proposition.

\begin{proposition}
\label{lemD}
Let $A=S/I$ be an Artinian graded $K$-algebra and let $p$ be an integer with $\dim_K A_p = \max \{ \dim_K A_k: k \in \ZZ_{\geq 0}\}$.
\begin{itemize}
\item[(i)]
If $A$ has the Sperner property then $I+ \mideal^p$ has the Rees property.
\item[(ii)]
If there is a nonnegative integer $s<p$ such that
$I$ has no generators of degree $s+1$ and the multiplication $\times y : A_s \to A_{s+1}$ is not injective for any linear form $y \in S$,
then the ideal $I+ \mideal^p$ is not $\mideal$-full.
\end{itemize}
\end{proposition}

\begin{proof}
(i)
Let $L=I+\mideal^p$ and let $J \supset L$ be a homogeneous ideal of $S$.
What we must prove is that $\mu(J) \leq \mu (L)$.
Since $S/I$ has the Sperner property,
we have
$$\dim_K(J/(I+\mideal J))=\mu(J/I) \leq \mu(L/I)=\dim_K(L/(I+\mideal L)).$$
Then, since $\mideal J \supset \mideal L$, we have
\begin{align*}
\mu(J)&=\dim_K(J/(I+\mideal J)) + \dim_K((I+\mideal J )/\mideal J)\\
&\leq \dim_K(L/(I+\mideal L)) + \dim_K((I+\mideal L )/\mideal L)=\mu(L),
\end{align*}
as desired.

(ii)
We prove that, for any $y \in \mideal$, one has
$$\mideal(I+\mideal^p):y \ne I + \mideal^p.$$
If $\deg y \geq 2$ then $\mideal(I+\mideal^p):y  \supset \mideal^{p-1}$ but
$I+\mideal^{p} \not \supset \mideal^{p-1}$.
Suppose that $y$ is a linear form.
Since the multiplication $\times y :A_s \to A_{s+1}$ is not injective,
there is a polynomial $f \not \in I$ of degree $s$ such that $fy \in I_{s+1}$.
Since $I$ has no generators of degree $s+1$, we have $I_{s+1} = (\mideal I)_{s+1}$.
Hence $f \in \mideal(I+\mideal^{s+1}):y$ but $f \not \in I+ \mideal^{s+1}$, which implies the desired property.
\end{proof}

\begin{remark}
\label{rem}
The ideal $I=I_{a,b,c,\alpha,\beta,\gamma}+\mideal^{s+1}$ in Theorem \ref{thm2} does not have the strong Rees property.
Indeed, Lemma \ref{lemC}(ii) says that the value of the Hilbert function of $K[x_1,x_2,x_3]/I_{a,b,c,\alpha,\beta,\gamma}$ has equal values in degrees $s$ and $s+1$.
Thus, $J = I_{a,b,c,\alpha,\beta,\gamma} + \mathfrak m^s $ contains $I$
but has the same number of minimal generators as $I$.
\end{remark}

\section{Ideals with the strong Rees property}

The purpose of this section is to prove the following result.

\begin{theorem}
\label{3.1}
Let $N \geq 5$ be an integer, $n \geq 4$ an even number, $d= \frac n 2 (N-2)$,
$S=K[x_1,\dots,x_n]$, and
$$I=(x_1^N,x_2^N,x_1^{N-2}x_2^{N-2})+(x_3^{N-1},\dots,x_n^{N-1}) \subset S.$$
Then the ideal $I+\mideal^{d+1}$ has the strong Rees property but is not $\mideal$-full.
\end{theorem}

To prove the above result, we need a combinatorial technique called the LYM property which appears in the theory of partially ordered sets.
We refer the readers to \cite{An} for basics on this theory.

Throughout the remainder of the paper, we assume that every partially ordered set (poset for short) is finite.
Let $P = \bigcup_i P_i$ be a ranked poset with $P_i=\{ a \in P: \rank a=i\}$.
Two elements $a,b \in P$ are {\em comparable} if $a>b$ or $a<b$ in $P$.
A subset $A \subset P$ is called an {\em antichain} of $P$ if no two elements in $A$ are comparable.
Let $\mathcal A(P)$ be the set of all antichains of $P$.
The poset $P$ is called {\em Sperner} if
$$\max \{ \#A : A \in \mathcal A(P)\} = \max\{ \# P_k: k \in \ZZ_{\geq 0}\}.$$
The Sperner property of posets and that of monomial Artinian $K$-algebras are related as follows:
For a monomial Artinian  $K$-algebra $A=S/I$,
the set $\mathcal M(A)=\bigcup_k \mathcal M_k(A)$ of all monomials in $S$ which are not in $I$
forms a poset with divisibility as its order.
Then it is known that the algebra $A$ is Sperner if and only if the poset $\mathcal M (A)$ is Sperner (see \cite[Proposition 2.41]{Book}).

Next, we recall the LYM property.
Let $P =\bigcup_{i=0}^s P_i$ be a ranked poset.
We say that $P$ has the {\em LYM property}
if, for any antichain $A$ of $P$, one has
$$
\sum_{a \in A} \frac 1 {\#P_{\rank a}} \leq 1.
$$
It is easy to see that the LYM property implies the Sperner property.
In fact, if $A \subset P$ is an antichain and $P$ has the LYM property, then one has
$$\# A \times \frac 1 {\max \{\# P_k: k \in \ZZ_{\geq 0}\} } \leq \sum_{a \in A} \frac 1 {\# P_{\rank a}} \leq 1,$$
which implies that $P$ is Sperner.
The LYM property has another characterization.
For a subset $V \subset P_k$, let
$$\nabla V=\{ b \in P_{k+1}: b > a \mbox{ for some }a \in V\}.$$
The poset $P =\bigcup_{i=0}^s P_i$ is said to have the {\em normalized matching property} if,
for any $k=0,1,\dots,s-1$ and for any subset $V \subset P_k$, one has
\begin{align}
\label{NMP}
\frac {\# V} {\# P_k} \leq \frac {\# \nabla V} {\# P_{k+1}}.
\end{align}
The following fact is standard in Sperner theory.
See \cite[Theorem 2.3.1]{An}.

\begin{lemma}
\label{3.2}
A ranked poset $P$ has the LYM property if and only if $P$ has the normalized matching property.
\end{lemma}

We say that a ranked poset $P$ is {\em log-concave} if $(\#P_i)^2 \geq (\#P_{i-1}) (\#P_{i+1})$ for all $i$.
The following result of Harper \cite{Ha}, which is reproved by Hsieh--Kleitman \cite{HK}, is known as the product theorem for the LYM property.

\begin{lemma}
\label{3.3}
Let $P$ and $Q$ be ranked posets.
If $P$ and $Q$ have the LYM property and are log-concave,
then so does their Cartesian product $P \times Q$.
\end{lemma}

We now consider monomial Artinian $K$-algebras.
We say that a monomial Artinian $K$-algebra $A$ has the LYM property if the poset $\mathcal M(A)$ has the LYM property.

\begin{lemma}
\label{3.4}
Let $A=S/I$ be a monomial Artinian $K$-algebra having the LYM property.
Suppose that there is an integer $p$ such that $\dim A_p > \dim A_k$ for all $k \ne p$.
If $J$ is an ideal of $A$ satisfying $\mu(J) = \dim_k (A_p)$ then $J = (I + \mathfrak m^p)/I$.
\end{lemma}

\begin{proof}
Let $J$ be an ideal of $A$ with $\mu(J)=\dim_K A_p$.
Since taking initial ideals does not decrease the number of generators, we may assume that $J$ is a monomial ideal.
Let $G(J)$ be the set of minimal monomial generators of $J$.
Then $G(J)$ forms an antichain of $\mathcal M(A)$.
Since $\mathcal M(A)$ has  the LYM property,
we have
$$1 = \frac {\# G(J)} {\dim_K A_p} \leq \sum_{u \in G(J)} \frac 1 {\dim_K A_{\deg u}} \leq 1.$$
This implies
$\deg u=p$ for all $u \in G(J)$
since $\dim_K A_p > \dim_K A_k$ for all $k \ne p$,
and $J=(I+\mideal^p)/I$.
\end{proof}

The properties of ideals in an Artinian $K$-algebra with the largest number of generators were studied in \cite{IW},
and essentially the same result was proved in \cite[Proposition 17]{IW}.

\begin{corollary}
\label{3.5}
Let $A=S/I$ be a monomial Artinian $K$-algebra having the LYM property
and let $p>0$ be an integer satisfying that $\dim_K A_p > \dim_K A_k$ for all $k <p$.
Then $I+\mideal^p$ has the strong Rees property.
\end{corollary}

\begin{proof}
Since $S/(I+\mideal^{p+1})$ has the LYM property by Lemma \ref{3.2},
by replacing $I$ with $I+\mideal^{p+1}$ if necessary, we may assume that $A$ satisfies the assumption of Lemma \ref{3.4}.
Then the statement follows in the same way as in the proof of Proposition \ref{lemD}(i).
\end{proof}

Now we prove the main result of this section.

\begin{proof}[Proof of Theorem \ref{3.1}]
About the ideal $I$,
the following facts were proved in \cite[Proposition 7.1.3]{BMMNZ}.
\begin{itemize}
\item[(a)] $\dim_K (S/I)_k <\dim_K (S/I)_{k+1}$ for $k=0,1,\dots,d$, where $ d= \frac n 2 (N-2)$.
\item[(b)] The multiplication $\times y: (S/I)_d \to (S/I)_{d+1}$ is not injective for any linear form $y \in S$.
\end{itemize}
By applying Proposition \ref{lemD}(ii) to the $K$-algebra $S/(I+\mideal^{d+2})$,
the conditions (a) and (b) imply that $I+ \mideal^{d+1}$ is not $\mideal$-full.
We prove that $I+\mideal^{d+1}$ has the strong Rees property.
By the condition (a) and Corollary \ref{3.5},
what we must prove is that $S/I$ has the LYM property.
We prove that the poset $\mathcal M(S/I)$ has the LYM property.

Let $A=K[x_1,x_2]/(x_1^N,x_2^N,x_1^{N-2}x_2^{N-2})$.
Observe
$$\mathcal M (S/I) = \mathcal M(A) \times \mathcal M (K[x_3,\dots,x_n]/(x_3^{N-1},\dots,x_n^{N-1})).$$
Since the poset $\mathcal M (K[x_3,\dots,x_r]/(x_3^{N-1},\dots,x_n^{N-1}))$
is isomorphic to a divisor lattice (see \cite[\S 1.4.2]{Book}), which has the LYM property and is log-concave (see \cite{Ha}),
by Lemma \ref{3.3}, what we must prove is that the poset
$\mathcal M(A)$ has the LYM property and is log-concave.
It is easy to see that $\mathcal M(A)$ is log-concave since its rank function forms the sequence
$$(1,2,3,\dots,N-1,N,N-1,N-2,\dots,5,4,2,0).$$
We prove that $\mathcal M(A)$ has the normalized matching property.
Observe that $\mathcal M(A)$ has rank $2(N-2)$.
Since $\mathcal M(A)$ and $\mathcal M (K[x_1,x_2]/(x_1^N,x_2^N))$
have the same components up to rank $2(N-2)-1$
and since $\mathcal M (K[x_1,x_2]/(x_1^N,x_2^N))$ has the normalized matching property,
to check the normalized matching property of $\mathcal M(A)$,
we only need to check the condition \eqref{NMP} for $k=2(N-2)-1$.
This is straightforward since
$$
\mathcal M_{2(N-2)-1}(A)=\{x_1^{N-1}x_2^{N-4},x_1^{N-2}x_2^{N-3},x_1^{N-3}x_2^{N-2},x_1^{N-4}x_2^{N-1}\}
$$
and
$$
\mathcal M_{2(N-2)}(A)=\{x_1^{N-1}x_2^{N-3},x_1^{N-3}x_2^{N-1}\}.
$$
Hence $\mathcal M(A)$ has the LYM property by Lemma \ref{3.2}.
\end{proof}

\end{document}